\documentclass[11pt,a4paper,reqno,twoside]{amsart}
\usepackage{amssymb,amstext,amsmath}
\usepackage{amsthm}
\usepackage{amsrefs}
\usepackage{mathtools}
\usepackage{epsfig}
\usepackage[utf8]{inputenc}
\usepackage{mdwlist}
\usepackage{xr-hyper}
\usepackage{cite}
\usepackage{thmtools}
\usepackage{thm-restate}
\usepackage{enumerate}
\usepackage{fancyhdr}
\usepackage{tikz}
\usepackage{tikz-3dplot}
\usetikzlibrary{calc}
\usepackage{emptypage}
\usepackage{xspace}
\usepackage{setspace}
\usepackage{hyperref}

\newcommand{\va}{{\boldsymbol a}}

\newcommand{\ve}{{\boldsymbol e}}
\newcommand{\vr}{{\boldsymbol r}}

\newcommand{\vt}{{\boldsymbol t}}
\newcommand{\vu}{{\boldsymbol u}}
\newcommand{\vv}{{\boldsymbol v}}
\newcommand{\vw}{{\boldsymbol w}}
\newcommand{\vx}{{\boldsymbol x}}
\newcommand{\vy}{{\boldsymbol y}}
\newcommand{\vz}{{\boldsymbol z}}

\newcommand{\vone}{{\boldsymbol 1}}

\newcommand{\vnull}{{\boldsymbol 0}}

\DeclarePairedDelimiter\ceil{\lceil}{\rceil}
\DeclarePairedDelimiter\floor{\lfloor}{\rfloor}

\newcommand{\R}{\mathbb{R}}

\newcommand{\N}{\mathbb{N}}
\newcommand{\Z}{\mathbb{Z}}

\newcommand{\symmetricconvexbodies}[1]{\mathcal{K}_{0}^{#1}}
\newcommand{\centroidconvexbodies}[1]{\mathcal{K}_{c}^{#1}}
\newcommand{\simplices}[1]{\mathcal{S}^{#1}}
\newcommand{\centroidsimplices}[1]{\mathcal{S}_{c}^{#1}}

\newcommand{\ehrhartsimplex}[1]{S_{#1}}

\newcommand{\successiveminimum}[2]{\lambda_{#1}\left(#2\right)} 

\newcommand{\latticeenumerator}[1]{\mathrm{G}\left(#1\right)}  
\newcommand{\interior}[1]{\mathrm{int}\left(#1\right)\,}  
\newcommand{\boundary}[1]{\mathrm{bd}\left(#1\right)\,} 
\newcommand{\conv}{\mathrm{conv}\,} 
\newcommand{\volume}[1]{\mathrm{vol}\left(#1\right)\,} 

\newcommand{\disuni}{\mathbin{\setbox0\hbox{$\bigcup$}\rlap{\copy0}\raise.3%
  \ht0\hbox to \wd0{\hfil$\cdot$\hfil}}}

\newcommand{\unitball}[1]{\bar{B}_{1}}

\newcommand{\centroid}[1]{\mathrm{c}(#1)}

\newcommand{\dd}{\mathrm{ d}} 

\newtheorem{lemma}{Lemma}[section]

\newtheorem{theorem}{Theorem}[section]
\newtheorem{proposition}{Proposition}[section]
\newtheorem{remark}{Remark}[section]

\newtheorem{conjecture}{Conjecture}
\newtheorem*{conjecture*}{Conjecture}

\newtheorem*{theorem*}{Theorem} 

\title{Lattice point inequalities for centered convex bodies}
\author{Sören Lennart Berg}
\author{Martin Henk}
\email{berg\{henk\}@math.tu-berlin.de}
\address{Technische Universität Berlin, Institut für Mathematik,
  Stra{\ss}e des 17. Juni 136, D-10623 Berlin}

\numberwithin{equation}{section}
\begin{document}

\begin{abstract}
    We study upper bounds on the number of lattice points for  convex bodies having their
    centroid at the origin. For the family of simplices as well as in
    the planar case we obtain best possible results.  For arbitrary
    convex bodies  we provide an upper bound, which extends the
    centrally symmetric case and which, in particular, shows that  the
    centroid assumption is indeed much more restrictive than an
    assumption on the number of interior lattice points even for the
    class of lattice polytopes.   
\end{abstract}

\maketitle

\section{Introduction} \label{section:introduction}

A classical as well as fundamental problem in the Geometry of Numbers is finding bounds on the number of
lattice (integral) points $\latticeenumerator{K}=\#(K \cap \Z^d)$ of a \emph{convex body} $K$, i.e., a compact convex
set in $\R^d$, under certain kinds of conditions. Considering
the class $\symmetricconvexbodies{d}$ of all $o$-\emph{symmetric} convex bodies, 
i.e.,  all convex bodies satisfying $K=-K$,
this problem has been settled by Minkowski\cite[p. 79]{MR0249269}
under the condition that $\#(\interior{K} \cap \Z^d)=1$, i.e.,  
the origin is the only interior lattice point of
$K$.  He proved 
\begin{equation} \label{eq:minkowski3d}
     \latticeenumerator{K} \leq 3^d, 
\end{equation} 
and also gave a better bound of $2^{d+1}-1$  for strictly convex
$o$-symmetric  bodies. The equality
case in \eqref{eq:minkowski3d} has been characterized by Draisma, Nill
and McAllister\cite{MR3022127}. 
Furthermore, it was pointed out by Betke et al. \cite{MR1194034} that
Minkowski's bounds can easily be extended to arbitrary
$o$-symmetric convex bodies via the {\em first successive
minimum} $\lambda_1(K)$ of $K$, which, for latter purposes,  we define
here for any convex body with $\vnull\in \interior{K}$:
\begin{equation*} 
\lambda_1(K) = \min\{ \lambda \in\R_{>0}: \lambda K \cap \Z^d \neq
\{\vnull\}\}.
\end{equation*}
With this notation Betke et al. \cite{MR1194034} showed for
$\vnull$-symmetric convex bodies 
\begin{equation}
\latticeenumerator{K} \leq \left\lfloor \frac{2}{\lambda_1(K)}+1\right\rfloor^d,
\label{eq:minkowski3d-extended}
\end{equation}
where $\lfloor x\rfloor$ is the largest integer not larger than $x$.

In order to bound the number of  lattice points of more general convex
bodies one certainly needs restrictions either on the class of convex
bodies or on the position of the origin.  For instance, even for
simplices with one interior lattice point there is no general upper
bound as the  family of triangles with vertices $(-m,-1)$,$(m,-1)$
and $(0,1/(m-1))$, $m>1$, shows.
\begin{figure}[ht]
  \centering
    \begin{tikzpicture}
    \coordinate (Origin)   at (0,0);
    \coordinate (XAxisMin) at (-3.5,0);
    \coordinate (XAxisMax) at (3.5,0);
    \coordinate (YAxisMin) at (0,-2.5);
    \coordinate (YAxisMax) at (0,1.5);
    \draw [thin, gray,-latex] (XAxisMin) -- (XAxisMax);
    \draw [thin, gray,-latex] (YAxisMin) -- (YAxisMax);

    \clip (-3.2,-2.2) rectangle (3.2,1.2);
    \pgftransformcm{1}{0}{0}{1.0}{\pgfpoint{0cm}{0cm}}

    \coordinate (v_one) at (3,-1);
    \coordinate (v_two) at (-3,-1);
    \coordinate (v_three) at (0,1/2);
    \draw[thin, fill=gray, fill opacity=0.2] (v_one) -- (v_two) -- (v_three) -- (v_one);

    \foreach \x in {-7,-6,...,7}{
      \foreach \y in {-7,-6,...,7}{
        \node[draw,circle,inner sep=1pt,fill] at (\x,\y) {}; 
      }
    }

\end{tikzpicture}
  
  \caption{}
  \label{figure:convex_body_no_symmetry}
\end{figure}

If we are dealing, however, only with lattice polytopes, i.e., all
vertices are integral, then there are bounds on the number of lattice
points (actually, on the volume) in terms of the (non-zero) number of interior lattice
points, see., e.g.,     Hensley \cite{MR688412}, 
Lagarias \& Ziegler \cite{MR1138580}, Pikhurko \cite{MR1996360,
  Pikhurko2},  Averkov \cite{MR2967480}.

Even in the case of lattice simplices having only  one  interior lattice
point  such an upper  bound has to be double exponential in
$d$ as shown by Perles, Wills and Zaks\cite{MR651251}. They presented
a simplex $S_1^d \subset \R^d$ with a single interior lattice point and
\[
    \latticeenumerator{S_1^d} \geq \frac{2}{6 (d-2)!} 2^{2^{d-a}},
\]
where $a = 0.5856\ldots$ is a constant. 
Averkov, Kr\"umpelmann and Nill\cite{MR3318147} 
proved that $S_1^d$ has
maximum volume among all lattice simplices having one interior lattice point. It remains an
open question if a similar result is true considering the number of lattice points instead
of the volume as conjectured by Hensley \cite{MR688412}.

In this work we show that the number of lattice points of a convex body $K$ with 
$\interior{K} \cap \Z^d = \{\vnull\}$ having its centroid at the
origin is (still -- cf.~\eqref{eq:minkowski3d}) at most
exponential in the dimension $d$. The \emph{centroid} or \emph{barycenter} of a 
Lebesgue measurable set $A \subset \R^d$ having positive Lebesgue
measure is defined as
\begin{equation*}
    \centroid{A} = \frac{1}{\volume{A}} \int_A \vx \dd \vx,
\end{equation*}
where $\volume{A}$ denotes the \emph{volume}, i.e the the ($d$-dimensional) Lebesgue measure of $A$,
and $\dd \vx$ integration with respect to the Lebesgue measure. 
Letting  $\centroidconvexbodies{d}$ denote
the class of all convex bodies $K$ with $\centroid{K} = \vnull$ we prove the following
result.

\begin{proposition} \label{theorem:generalbound}
    Let $K \in \centroidconvexbodies{d}$.  
    Then
    \begin{equation} \label{eq:theoremgeneralbound}
        \latticeenumerator{K} < 2^d \left( \frac{2}{\lambda_1(K)} + 1 \right)^d.
    \end{equation}
\end{proposition}
Particularly, if $\lambda_1(K) \geq 1$, i.e., $\interior{K} \cap \Z^d = \{ \vnull \}$, 
Proposition \ref{theorem:generalbound} yields $\latticeenumerator{K} <
6^d$.  

The bound in Proposition \ref{theorem:generalbound} is quite likely
asymptotically not sharp and
we believe that the worst case is attained by $d$-simplices. More
precisely, let $\ve_i\in\R^d$ be the $i$th unit vector, $\vone = (1, \ldots, 1)$ the all ones vector
 and let $\conv{A}$ be the \emph{convex hull} of a given
point set $A\subset\R^d$.  Let 
\begin{equation} 
\begin{split}
   \ehrhartsimplex{d} &:= (d+1)\conv\{\vnull,\ve_1, \ldots, \ve_d\}
   - \vone\\ 
&=\{\vx\in\R^d: x_i\geq -1, 1\leq i\leq d, x_1+x_2+\cdots + x_d\leq
1\}.
\label{eq:ehrhartsimplex}
\end{split}
\end{equation}
\begin{figure}[ht]
    \centering
        \begin{tikzpicture}
    \coordinate (Origin)   at (0,0);
    \coordinate (XAxisMin) at (-2.5,0);
    \coordinate (XAxisMax) at (3.5,0);
    \coordinate (YAxisMin) at (0,-2.5);
    \coordinate (YAxisMax) at (0,3.5);
    \draw [thin, gray,-latex] (XAxisMin) -- (XAxisMax);
    \draw [thin, gray,-latex] (YAxisMin) -- (YAxisMax);

    \clip (-2.2,-2.2) rectangle (3.2,3.2);
    \pgftransformcm{1}{0}{0}{1.0}{\pgfpoint{0cm}{0cm}}

    \coordinate (v_one) at (-1,-1);
    \coordinate (v_two) at (2,-1);
    \coordinate (v_three) at (-1,2);
    \draw [thin, fill=gray, fill opacity=0.2] (v_one) -- (v_two) -- (v_three) -- (v_one);

    \foreach \x in {-7,-6,...,7}{
      \foreach \y in {-7,-6,...,7}{
        \node[draw,circle,inner sep=1pt,fill] at (\x,\y) {};
      }
    }

\end{tikzpicture}
  
    \caption{$\ehrhartsimplex{2}$}
    \label{figure:ehrhartsimplex2d}
\end{figure}

For an integer $m\in\N$ it is easily verified that
$\latticeenumerator{m\,\ehrhartsimplex{d}}=\binom{d+m(d+1)}{d}$ and
 we  believe  that this type of bound is the correct upper bound in
Proposition 
\ref{eq:theoremgeneralbound}, i.e.,   
\begin{conjecture} \label{conjecture:generalbound}
     Let $K \in \centroidconvexbodies{d}$. Then
     \begin{equation*}
         \latticeenumerator{K} \leq \binom{d+\ceil*{ \lambda_1(K)^{-1} (d+1) } }{d}.
     \end{equation*}
 \end{conjecture} 
Observe, compared to \eqref{eq:theoremgeneralbound} this bound is
asymptotically  smaller by a factor of $(\mathrm{e}/4)^d$.  
We will verify  this conjecture for arbitrary simplices. 
\begin{theorem} \label{theorem:simplexbound}
    Let $S \in \centroidconvexbodies{d}$ be a $d$-simplex. 
    Then
    \begin{equation*}
        \latticeenumerator{S} \leq \binom{d+\ceil*{ \lambda_1(S)^{-1}  (d+1)}}{d}.
    \end{equation*}
    Furthermore, for $\lambda_1(S)^{-1}\in\N$ equality holds if and
    only if  $S$ is unimodularly equivalent
    to $\lambda_1(S)^{-1}\,\ehrhartsimplex{d}$.
\end{theorem}
Here  we say that
two sets $A, B \subset \R^d$ are \emph{unimodularly equivalent} if 
$A = UB + \vz$ for a unimodular matrix $U \in \Z^{d \times d}$, i.e.,
$|\det U|=1$, and $\vz \in \Z^d$.

It is noteworthy that Conjecture \ref{conjecture:generalbound}
is strongly related to 
a well-known conjecture by Ehrhart.

\begin{conjecture*}[Ehrhart, \cite{MR0163219}]
    Let $K \subset \R^d$ be a convex body with $\centroid{K} =
    \vnull$, and $\interior{K}\cap\Z^d=\{\vnull\}$. Then 
\begin{equation*}
\volume{K}\leq \volume{S_d}=\frac{(d+1)^d}{d!}.
\end{equation*}
\end{conjecture*}
In fact, due to the Jordan measurability  of convex bodies (cf., e.g.,
\cite[Theorem 7.4]{MR2335496}),
Conjecture \ref{conjecture:generalbound} implies Ehrhart's conjecture:

\begin{align*}
    \volume{K} 
    & = \lim_{m\to\infty } m^{-d} \#\left(K \cap \frac{1}{m} \Z^d \right)
    = \lim_{m\to\infty}  m^{-d}\#\left(m\,K \cap \Z^d \right) \\
    & \leq \lim_{m\to\infty} m^{-d} \binom{d+\ceil*{ \successiveminimum{1}{m\,K}^{-1}  (d+1) } }{d} \\
    & = \lim_{m\to\infty} m^{-d} \binom{d+\ceil*{ m \successiveminimum{1}{K}^{-1}  (d+1) } }{d} \\
    & \leq \lim_{m\to\infty} m^{-d} \binom{d+{ m \successiveminimum{1}{K}^{-1}  (d+1) +1 } }{d}
    = \successiveminimum{1}{K}^{-d} \frac{(d+1)^d}{d!}.
\end{align*}
Hence, if $\interior{K}\cap\Z^d=\{\vnull\}$ we have
$\successiveminimum{1}{K}\ge 1$ and Conjecture
\ref{conjecture:generalbound} gives Ehrhart's conjecture. 
For recent progress regarding the latter we refer to \cite{2012arXiv1204.1308B}
and \cite{MR3276123}. Ehrhart proved his conjecture in the 
plane \cite{MR0066430, MR0163219} and for 
simplices in any dimension\cite{MR518864}. Observe, that Theorem
\ref{theorem:simplexbound} also implies Ehrhart's result for simplices
(cf.~\cite[Proposition 2.15]{HHZ}).   
The problem is also briefly discussed
in \cite[p. 147]{MR1316393}. It is worth mentioning that $\ehrhartsimplex{d}$ has maximal
volume among all lattice simplices having the centroid at the origin, see Lemma 5 in \cite{MR1996360}.
Moreover, every such lattice simplex has also constant Mahler volume 
\cite[Theorem 2.4, Proposition 6.1]{MR3318147}.

Finally, by using classical results of planar geometry,  our last
result verifies Conjecture \ref{conjecture:generalbound}  for planar convex bodies
whose only lattice point is the origin.
\begin{theorem} \label{theorem:planarcase}
    Let $K \in \centroidconvexbodies{2}$ with $\latticeenumerator{\interior{K}}=1$. Then
     \begin{equation*}
         \latticeenumerator{K} \leq 10.
     \end{equation*}
     Furthermore, equality holds if and only if $K$ is unimodularly equivalent to $\ehrhartsimplex{2}$.
\end{theorem}

This paper is organised as follows. In Section \ref{section:generalbound} we will discuss 
the proof of Proposition \ref{theorem:generalbound} in detail. 
The proof of Theorem \ref{theorem:simplexbound} is presented in Section \ref{section:simplexbound}.
It relies mainly on the fact that a point in a simplex can be described in a unique way by its barycentric coordinates. 
In Section \ref{section:planarcase}
we will provide the proof of Theorem \ref{theorem:planarcase}, which is based on
results due to Ehrhart, Gr\"unbaum, Pick and Scott.

\section{The Proof of Proposition \ref{theorem:generalbound}} \label{section:generalbound}

Considering a convex body $K$ it appears plausible that its centroid $\centroid{K}$
is located deep inside $K$, i.e., not too close to the 
boundary of $K$. In conclusion one would expect that the volume of the intersection
of $K$ with its  reflection  
at  $\centroid{K}$  cannot be too small with respect
to the volume of $K$. Indeed Milman and Pajor proved the following result 
on which our proof relies heavily. 

\begin{theorem}[\hspace{1sp}\protect{\cite[Corollary 3]{MR1764107}}] \label{thm:MilmanPajor}
    Let $K \in \centroidconvexbodies{d}$ with $\centroid{K} = \vnull$. Then
    \begin{equation*}
        \volume{K \cap -K} \geq 2^{-d} \volume{K}.
    \end{equation*}
\end{theorem}

\begin{proof}[Proof of Proposition \ref{theorem:generalbound}]

    Let $K \in \centroidconvexbodies{d}$ and for short we set
    $\lambda_1 = \successiveminimum{1}{K}$. First we observe that 
    $\frac{\lambda_1}{2}(K \cap -K)$ is a packing set with respect to
  the integer lattice, i.e.,  we have
    \begin{equation*}
    		\interior{\vu+\frac{\lambda_1}{2}(K \cap -K)} 
    		    \cap \interior{\vv+\frac{\lambda_1}{2}(K \cap -K)} 
    		    = \emptyset,
    \end{equation*}
     for $\vu, \vv \in \Z^d$, $\vu\ne\vv$; otherwise, by the
     $o$-symmetry of $K\cap -K$ we get 
\begin{equation*}
\begin{split}
 \vu-\vv&\in \frac{\lambda_1}{2} \interior{K\cap-K}+\frac{\lambda_1}{2}
     \interior{K\cap-K} \\ &=\lambda_1\, \interior{K\cap-K} 
  \subseteq \lambda_1\interior{K},
\end{split}
\end{equation*}
contradicting the minimality of $\lambda_1$. We also certainly have
(cf.~Figure \ref{figure:covering_volume_bound_tikz})

    \begin{equation} \label{eq:translatetrivial}
    	(\Z^d \cap K) + \tfrac{\lambda_1}{2}(K \cap -K) \subseteq K +
        \tfrac{\lambda_1}{2}(K \cap -K)\subseteq
        \left(1+\frac{\lambda_1}{2}\right)K. 
    \end{equation}
    Hence,  in view of Theorem \ref{thm:MilmanPajor} we find

    \begin{equation} \label{eq:latticeboundvolume}
\begin{split}
        \latticeenumerator{K} 
            \leq \frac{ \volume{\left(1+\frac{\lambda_1}{2}\right)K}
            }{ \volume{\tfrac{\lambda_1}{2}(K \cap -K)}
            }&=\left(\frac{2}{\lambda_1}+1\right)^d\frac{\volume{K}}{\volume{K\cap
             {-K}}}\\ &\leq 2^d\left(\frac{2}{\lambda_1}+1\right)^d.
\end{split}
    \end{equation}
If $K$ is $o$-symmetric we know by \eqref{eq:minkowski3d-extended}  that the last inequality
is strict, and if $K$ is not symmetric with respect to $\vnull$ we
have a strict inclusion in the last inclusion of \eqref{eq:translatetrivial}.
\end{proof}

\begin{figure}[ht]
    \begin{tikzpicture}
    \coordinate (Origin)   at (0,0);
    \coordinate (XAxisMin) at (-2.5,0);
    \coordinate (XAxisMax) at (3.5,0);
    \coordinate (YAxisMin) at (0,-2.5);
    \coordinate (YAxisMax) at (0,3.5);
    \draw [thin, gray,-latex] (XAxisMin) -- (XAxisMax);
    \draw [thin, gray,-latex] (YAxisMin) -- (YAxisMax);

    \clip (-2.2,-2.2) rectangle (3.2,3.2);
    \pgftransformcm{1}{0}{0}{1.0}{\pgfpoint{0cm}{0cm}}
    
    \coordinate (v_one) at (-1,-1);
    \coordinate (v_two) at (2,-1);
    \coordinate (v_three) at (-1,2);
    
    \coordinate (w_one) at (-3/2,-1);
    \coordinate (w_two) at (-1,-3/2);
    \coordinate (w_three) at (5/2,-3/2);
    \coordinate (w_four) at (5/2,-1);
    \coordinate (w_five) at (-1,5/2);
    \coordinate (w_six) at (-3/2,5/2);
    
    \foreach \v in {(-1,-1),(-1,0),(-1,1),(-1,2),(0,-1),(0,0),(0,1),(1,-1),(1,0),(2,-1)}
    {
        \draw[dashed, fill=gray, fill opacity=0.1] [shift={\v}] (-1/2,1/2) -- (0,1/2) -- (1/2,0) -- (1/2,-1/2) -- (0,-1/2) -- (-1/2,0) -- (-1/2,1/2);
    }
    
    \draw[thin, dashed, fill=gray, fill opacity=0.0] (w_one) -- (w_two) -- (w_three) -- 
                (w_four) -- (w_five) -- (w_six) -- (w_one);
    
    \draw [thin, fill=gray, fill opacity=0.0] (v_one) -- (v_two) -- (v_three) -- (v_one);

    \foreach \x in {-7,-6,...,7}{
      \foreach \y in {-7,-6,...,7}{
        \node[draw,circle,inner sep=1pt,fill] at (\x,\y) {};
      }
    }

\end{tikzpicture}
    \caption{}
    \label{figure:covering_volume_bound_tikz}
\end{figure}

\begin{remark}
    In the case that the convex body $K$ is centrally symmetric the
    proof above gives essentially  the  result \eqref{eq:minkowski3d-extended}
    since then $K \cap -K = K$  in \eqref{eq:latticeboundvolume}. In
    particular, it also recovers Minkowski's result
    \eqref{eq:minkowski3d} which he proved by a simple residue class
    argument.  Actually, by such an argument Minkowski also showed  
    that in the case of strictly $o$-symmetric convex
    bodies $K$ with $\lambda_1(K)=1$ the stronger bound holds true 
\begin{equation*}
                 \latticeenumerator{K}\leq 2^{d+1}-1.
\end{equation*}
In fact, this bound is even true without symmetry solely under the
assumption $\interior{K}\cap\Z^d=\{\vnull\}$. 
\end{remark}

Proposition \ref{theorem:generalbound} is certainly not sharp, since  for non $o$-symmetric
convex bodies the left-hand side in \eqref{eq:translatetrivial} is a proper subset of
the right-hand side (see also in Figure \ref{figure:covering_volume_bound_tikz}). 
In addition,  it is  unknown whether Corollary \ref{thm:MilmanPajor} 
is sharp itself. In fact, it seems rather plausible that 
the volume ratio $\volume{K \cap -K}/\volume{K}$
is minimal if $K$ is a simplex for which it is exponentially greater
than $2^{-d}$, namely (roughly) of order  $(2/\mathrm{e})^d$.

For $d=2$ this problem has been solved by Stewart, c.f.~\cite{MR0097771}. 
For recent results regarding this problem see \cite{MR3300714}.

\section{Proof of Theorem \ref{theorem:simplexbound}} \label{section:simplexbound}

Let  $S=\conv\{\vv_1,\dots,\vv_{d+1}\}\subset \R^d$ be a
$d$-dimensional simplex, and for $\vx \in \R^d$ let $\beta_S(\vx)\in\R^{d+1}$ its
unique {\em  affine (barycentric)} coordinates with respect to the vertices of
$\vv_1,\dots,\vv_{d+1}$ of $S$, i.e., 
\begin{equation*}
   \vx =  \sum_{i=1}^{d+1} \beta_S(\vx)_i \vv_i \quad \text{ and  } \quad
     \sum_{i=1}^{d+1} \beta_S(\vx)_i = 1.
\end{equation*}
Moreover, we have $\vx\in S$ if and only if
$\beta_S(\vx)\in\R^{d+1}_{\geq 0}$, i.e., all entries are non-negative,
as well as 
$\vx \in \interior{S}$ if and only if $\beta_S(\vx)\in\R^{d+1}_{>
  0}$. For $\vx,\vy$ and $\lambda,\mu\in\R$ we have 
\begin{equation*}
\beta_S(\mu\,\vx+\lambda\,\vy)=\mu\beta_S(\vx)+\lambda\beta_S(\vy)+(1-(\mu+\lambda))\,\beta_S(\vnull). 
\end{equation*}

We denote by $\simplices{d}$ the class of all $d$-dimensional simplices and by
$\centroidsimplices{d}$ the set of all $S \in \simplices{d}$ with
$\centroid{S}=\vnull$.  Observe, for $S\in\centroidsimplices{d}$ we
have 
\begin{equation*}
  \beta_S(\vnull)_i=\frac{1}{d+1}, \,1\leq i\leq d+1.
\end{equation*}

\begin{lemma} \label{lemma:barycentricinequality}
    Let $S \in \centroidsimplices{d}$,   and let $\vu,\vw \in S \cap
    \Z^d$, $\vu\ne\vw$.  Then there exists an index
    $k\in\{1,\dots,n+1\}$ with 
\begin{equation*} 
 \beta_S(\vu)_k-\beta_S(\vw)_k\geq \lambda_1(S)\,\frac{1}{d+1}.
\end{equation*}
\end{lemma}
\begin{proof} Suppose the opposite, i.e.,  $\beta_S(\vu)_i -
  \beta_S(\vw)_i < \lambda_1/(d+1)$ for all $1\leq i\leq d+1$, and let
  $\vv=\vw-\vu$. Then with $\lambda_1=\lambda_1(S)$ we get 
\begin{equation*}
\beta_S\left(\frac{1}{\lambda_1}\vv\right) = 
 \frac{1}{\lambda_1}\left(\beta_S(\vw)-\beta_S(\vu)\right)+\beta_S(\vnull)>0. 
\end{equation*}
Hence, the non-trivial lattice point $\vv$ 
belongs to $\interior{\lambda_1\,S}$, and thus
contradicting the minimality of $\lambda_1$.
\end{proof}

Obviously,  Lemma \ref{lemma:barycentricinequality} says, that if two lattice points
$\vu$ and $\vw$ in a given simplex are located close to each other, that is, the difference 
of their barycentric coordinates is small, then $\vu-\vw$ will lie inside the simplex as pictured in 
Figure \ref{figure:lemma_barycentric_bound_tikz}.

\begin{figure}[ht]
    \begin{tikzpicture}
    \coordinate (Origin)   at (0,0);
    \coordinate (XAxisMin) at (-2.5,0);
    \coordinate (XAxisMax) at (2,0);
    \coordinate (YAxisMin) at (0,-1.5);
    \coordinate (YAxisMax) at (0,2.5);
    \draw [thin, gray,-latex] (XAxisMin) -- (XAxisMax);
    \draw [thin, gray,-latex] (YAxisMin) -- (YAxisMax);

    \clip (-2.2,-2.2) rectangle (3.2,3.2);
    \pgftransformcm{1}{0}{0}{1.0}{\pgfpoint{0cm}{0cm}}
    
    \coordinate (v_one) at (-1.6,-0.6);
    \coordinate (v_two) at (1.6,-1.2);
    \coordinate (v_three) at (0,1.8);
    \draw [thin, fill=gray, fill opacity=0.2] (v_one) -- (v_two) -- (v_three) -- (v_one);
    
    \filldraw (0.4,0.6) circle (2pt) node[anchor=south west]{$\vu-\vw$};
    \filldraw (-1.4,-0.3) circle (2pt) node[anchor=east]{$\vw$};
    \filldraw (-1,0.3) circle (2pt) node[anchor=east]{$\vu$};

\end{tikzpicture}
  
    \caption{}
    \label{figure:lemma_barycentric_bound_tikz}
\end{figure}

We introduce some more notation.
Let
\begin{equation*}
        B = \left\{ \vx \in \R^{d+1}_{\geq 0}: \sum_{i=1}^{d+1} x_i = 1 \right\},
\end{equation*}
which we regard as the $d$-dimensional simplex in $\R^{d+1}$ of all
feasible barycentric coordinates of points contained in a
$d$-dimensional simplex. 
For a given real number $\rho>0$ let 
\begin{equation}
     n(\rho)= \ceil*{\rho^{-1}(d+1)}
\label{eq:rho} 
\end{equation}
and 
\begin{equation*}
    R_\rho = \frac{1}{n(\rho)} \left\{ \va \in \Z^{d+1}_{\geq 0}: \sum_{i=1}^{d+1} a_i = n(\rho) \right\}.
\end{equation*}

Note that $\#(R_\rho )= \binom{d+n(\rho) }{d}$ and $R_\rho\subset
B$. Let $Z_\rho= \left[0, \frac{1}{n(\rho)}  \right)^d \times \R\subset\R^{d+1}$ be
the cylinder over the half-open $d$-dimensional cube of edge length
$\frac{1}{n(\rho)}$.  For $\vr\in R_\rho$ the intersection of $\vr+Z_\rho$ with the affine space $\{ \vx \in \R^{d+1}: \sum x_i =1 \}$
containing $B$ yields a half-open $d$-dimensional parallelepiped, as
depicted  in Figure \ref{figure:barycentric_box_tikz}.
\begin{figure}[ht]
    \tdplotsetmaincoords{70}{110}
\begin{tikzpicture}[tdplot_main_coords]
\def\RI{2}
\def\RII{1.25}

\coordinate (e1) at (2,0,0);
\coordinate (e2) at (0,2,0);
\coordinate (e3) at (0,0,2);

\coordinate (ax1) at (3,0,0);
\coordinate (ax2) at (0,3,0);
\coordinate (ax3) at (0,0,3);

\coordinate (O) at (0,0,0);

\coordinate (r) at (0,2,0);

\coordinate (box1) at (0,0,2);
\coordinate (box2) at (0,2/4,6/4);
\coordinate (box3) at (2/4,2/4,4/4);
\coordinate (box4) at (2/4,0,6/4);


%
\draw[->] (e1) -- (ax1) node[anchor=north east]{$x_1$};
\draw[->] (e2) -- (ax2) node[anchor=north west]{$x_2$};
\draw[->] (e3) -- (ax3) node[anchor=south]{$x_3$};

\draw[dashed] (O) -- (e1);
\draw[dashed] (O) -- (e2);
\draw[dashed] (O) -- (e3);

\node[draw,circle,inner sep=1pt,fill,shift={(1/8,3/8,-1/2)}] at (0,0,2) {};
\draw[shift={(1/8,3/8,-1/2)}, thin] (0,0,2) -- (0,2/4,6/4);
\draw[shift={(1/8,3/8,-1/2)}, thin ] (2/4,0,6/4) -- (0,0,2);

\filldraw[shift={(1/8,3/8,-1/2)}, gray, opacity=0.0, fill=gray, fill opacity=0.2] (0,0,2) -- (0,2/4,6/4) -- (2/4,2/4,4/4)-- (2/4,0,6/4) -- (0,0,2);

\draw[thin] (e1) -- (e2);
\draw[thin] (e2) -- (e3);
\draw[thin] (e3) -- (e1);

\end{tikzpicture}
    \caption{}
    \label{figure:barycentric_box_tikz}
\end{figure}

Next we claim 
\begin{lemma} With the notation above we have $B\subset
  R_\rho+Z_\rho$.
\label{lem:two}
\end{lemma} 
\begin{proof} We may write  
\begin{equation} 
     R_\rho+Z_\rho = \bigcup_{ \vr\in R_\rho}\left\{ \vx \in \R^{d+1}: r_i \leq x_i < r_i +
       \frac{1}{n(\rho)}, i=1,\ldots,d\right\}.
\label{eq:union}
\end{equation}
For a given  $\vx \in B$ let 
$\vr \in \R^{d+1}$ be defined as 
         \begin{align}
         \begin{split} \label{eq:defpointr}
             r_i &=  \frac{ \floor*{ n(\rho) x_i } }{ n(\rho) }
                      \  \text{ for } 1\leq i\leq d , \\
             r_{d+1} &= 1 -  \sum_{i=1}^{d}r_i 
                      = \frac{ n(\rho) 
                             - \sum_{i=1}^d \floor*{ n(\rho)x_i } }
                             { n(\rho) }.
         \end{split}
         \end{align}
Obviously $r_i\geq 0$, $1\leq i\leq d$, $\sum_{i=1}^d r_i=1$ and 
\begin{equation*}
       r_{d+1}= \frac{ n(\rho) 
                             - \sum_{i=1}^d \floor*{ n(\rho)x_i } }
                             { n(\rho) } \geq 1-\sum_{i=1}^d x_i
                             =x_{d+1}\geq 0.
\end{equation*}   
Hence, $\vr\in R_\rho$. We also have  $r_i \leq x_i$, $1\leq i\leq d$,
as well as 
\begin{equation*}
   x_i-r_i=  \frac{n(\rho)x_i-\floor*{ n(\rho) x_i }}{n(\rho)}\leq \frac{1}{n(\rho)}.
\end{equation*}
Thus we have shown $\vx\in\vr+Z_\rho$.
\end{proof}
Although not needed for the proof we remark that the union in \eqref{eq:union} is
disjoint. 

\begin{proof}[Proof of Theorem \ref{theorem:simplexbound}]
    Let $S\in \centroidsimplices{d}$ be a $d$-simplex and for short we write
    $\lambda_1=\lambda_1(S)$. Suppose  
\begin{equation*}
\latticeenumerator{S} > \binom{d+ 
								n(\lambda_1)}{d}= \#(R_{\lambda_1}).
\end{equation*}
According to Lemma \ref{lem:two} the set of all barycentric
    coordinates $B$ of points in $S$ is covered by
    by the cylinders $R_{\lambda_1}+Z_{\lambda_1}$.  Hence there
    exist an $\vr\in R_{\lambda_1}$ and two lattice points $\vu,\vw\in
    S$ such that 
\begin{equation*}
   \beta_S(\vu),\beta_S(\vw)\in \vr + Z_{\lambda_1}.
\end{equation*} 
We may assume $\beta_S(\vu)_{d+1}-\beta_S(\vw)_{d+1}< 1/n(\lambda_1)$ and
due to the definition of $Z_{\lambda_1}$ we also have
$|\beta_S(\vu)_{i}-\beta_S(\vw)_{i}|< 1/n(\lambda_1)$, $1\leq i\leq
d$.  Thus 
\begin{equation*}
\beta_S(\vu)_{i}-\beta_S(\vw)_{i}< \frac{1}{n(\lambda_1)}\leq \lambda_1\frac{1}{d+1}  
\end{equation*}
contradicting Lemma \ref{lemma:barycentricinequality}.

Next we discuss the equality case. 
Let $\lambda_1^{-1} \in \N$ and let 
$\latticeenumerator{S} = \binom{d+n(\lambda_1)}{d}$. 
We first show that the set $\beta_S(S \cap \Z^d)$, consisting of the barycentric coordinates
of all lattice points in $S$, equals $R_{\lambda_1}$.
To this end, suppose there exists
$\vx \in X = \beta_S(S \cap \Z^d) \setminus R_{\lambda_1}$. 
Thus, there exists an index $\ell$ such that $x_\ell \neq \frac{k}{n(\lambda_1)}$
for every integer $0 \leq k \leq n(\lambda_1)$. Let
$\vx$  be such a point in $X$ having maximal $x_\ell$.
We may assume $\ell \neq d+1$ and define
$\vr$ as in \eqref{eq:defpointr}, where $\rho = \lambda_1$.
Accordingly it holds $\vr \in \beta_S(S \cap \Z^d)$. Since $r_\ell < x_\ell$,
we have $r_{d+1} > 0$ and thus $r_{d+1} \geq \frac{1}{n(\lambda_1)}$.
Let
\begin{align*}
	t_i &=  r_i	\  \text{ for } 1\leq i\leq d , i \neq \ell, \\
	t_\ell &= r_\ell + \frac{1}{n(\lambda_1)} = \frac{ \floor*{ n(\rho) x_i } + 1}{ n(\lambda_1) }, \\
	t_{d+1} &= r_{d+1} - \frac{1}{n(\lambda_1)}.
\end{align*}
Then $\vt \in R_{\lambda_1}$. Let $\vv \in S \cap \Z^d$ be the unique lattice point
in $S$ such that $\beta_S(\vv) \in \vt + Z_{\lambda_1}$. 
Then $\beta_S(\vv)_\ell = t_\ell$ by the assumption on $\vx$.
We now conclude

\begin{align*}
	x_i - \beta_S(\vv)_i & \leq x_i - t_i \leq x_i - r_i < \frac{1}{n(\lambda_1)}, 
		 \ 1 \leq i \leq d, \\
	\beta_S(\vv)_i - x_i & < t_i + \frac{1}{n(\lambda_1)} - x_i
	   = r_i + \frac{1}{n(\lambda_1)} - x_i \leq \frac{1}{n(\lambda_1)}, \ i \notin \{\ell, d+1\}, \\
	\beta_S(\vv)_\ell - x_\ell &= t_\ell - x_\ell = 
	r_\ell + \frac{1}{n(\lambda_1)} - x_\ell < \frac{1}{n(\lambda_1)}.
\end{align*}
Clearly, $x_{d+1} - \beta_S(\vv)_{d+1} < \frac{1}{n(\lambda_1)}$ or 
$\beta_S(\vv)_{d+1} - x_{d+1} < \frac{1}{n(\lambda_1)}$. Let $\vz$ be the lattice
point in $S$ such that $\beta_S(\vz) = \vx$. Then $\vz$ and $\vv$ contradict
Lemma \ref{lemma:barycentricinequality}, and thus we have $\beta_S(S \cap \Z^d) = R_{\lambda_1}$.

We now show that $S$ and $\lambda_1^{-1}\ehrhartsimplex{d}$ are
unimodularly equivalent. Let $S = \conv\{\vv_1, \vv_2, \ldots, \vv_{d+1}\}$. Since
$\beta_S(S \cap \Z^d) = R_{\lambda_1}$ there are exactly $n(\lambda_1)+1$ lattice points
on every edge of $S$. Therefore $\vw_i = \frac{1}{n(\lambda_1)}(\vv_i - \vv_{d+1}) \in \Z^d$
for all $i$. Moreover, by Lemma 5 in \cite{MR1996360} it holds 
$\volume{\lambda_1 S} \leq \volume{\ehrhartsimplex{d}}$ and letting $M$ be the $d \times d$
Matrix having columns $\vw_1, \vw_2, \ldots, \vw_{d}$ we conclude
\begin{equation*}
	|M| = \frac{d!}{n(\lambda_1)^d} \volume{S} 
	\leq \frac{d!}{n(\lambda_1)^d} \volume{\lambda_1^{-1} \ehrhartsimplex{d}} = 
	\frac{d!}{(d+1)^d} \volume{\ehrhartsimplex{d}} = 1.
\end{equation*}
Thus, $M$ is unimodular and the equation
\begin{equation*}
	M(\lambda_1^{-1} \ehrhartsimplex{d} + \lambda_1^{-1} \vone) + \vv_{d+1} = S,
\end{equation*}
implies that $S$ and $\lambda_1^{-1}\ehrhartsimplex{d}$ are indeed
unimodularly equivalent.
\end{proof}

\section{The Proof of Theorem \ref{theorem:planarcase}} \label{section:planarcase}

Let $\boundary{K}$ denote the boundary of a set $K$. 
Pick proved the following remarkable  identity  for lattice polygons.  
\begin{theorem}[Pick, \cite{Pick}]
    Let $P$ be a lattice polygon in the plane. Then
    \begin{equation*}
        \latticeenumerator{P} = \volume{P} + \frac{1}{2}\latticeenumerator{\boundary{P}} + 1.
    \end{equation*}
\end{theorem}
Pick's theorem has been generalised to higher dimensions by Ehrhart, cf. 
\cite{MR0130860,MR0213320,MR0234919}, and for a modern introduction to the
underlying theory we refer to \cite{MR2271992}. 
Scott \cite{MR0430960} stated the following result having a very similar flavor.
\begin{theorem}[Scott, \cite{MR0430960}]
    Let $P$ be a convex lattice polygon with at least one interior point. Then 
\begin{equation*}
    \latticeenumerator{\boundary{P}}-2\latticeenumerator{\interior{P}}
\leq 7, 
\end{equation*} 
 and  equality is attained if and only if $P$ is unimodularly  equivalent
    to $\ehrhartsimplex{2}$.  
\end{theorem}

One of the most fascinating theorems regarding convex bodies and their centroids was presented
by Gr\"unbaum. It provides an interesting property of the centroid in terms of mass-distribution
of the given convex body.
\begin{theorem}[Gr\"unbaum, \cite{MR0124818}]
    Let $K \in \centroidconvexbodies{d}$ be convex body and let $H
    \subset \R^d$ be a half-space containing the centroid $c(K)$,  then
    \begin{equation*}
        \volume{K \cap H} \geq \left( \frac{d}{d+1} \right)^d \volume{K}.
    \end{equation*}
\end{theorem}
It is noteworthy that equality holds in Gr\"unbaum's theorem if $K$ is a simplex.

In order to prove Theorem \ref{theorem:planarcase} we will also use the following theorem by
Ehrhart discussed already in Section \ref{section:introduction}.
\begin{theorem}[Ehrhart, \cite{MR0066430, MR0163219}]
    Let $K \in \centroidconvexbodies{2}$ with $\volume{K}\geq 9/2$. Then $K$ contains
    at least two lattice points distinct from the origin.
\label{thm:ehrhart}
\end{theorem}

\begin{proof}[Proof of Theorem \ref{theorem:planarcase}] 
    Let $P = \conv({K \cap \Z^2})$. If $\vnull \notin \interior{P}$ then there exists
    a half-space $H$ containing $P$ and containing $\vnull$ in its boundary such that
    $K \cap H \cap \Z^2 = K \cap \Z^2$. Gr\"unbaum's Theorem and Ehrhart's Theorem imply that
    \begin{equation*}
            \volume{P}
            \leq \volume{K \cap H} 
            \leq \left(1 - \left(\frac{2}{3}\right)^2\right) \volume{K}
            = \frac{5}{9}\volume{K}
            \leq \frac{5}{2}.
    \end{equation*}
    In turn applying Pick's Theorem yields
    \begin{align*}
        \latticeenumerator{K} & = \latticeenumerator{K \cap H} = \latticeenumerator{P}  
            =\volume{P} + \frac{1}{2} \latticeenumerator{\boundary{P}} + 1 
            \leq \frac{7}{2} + \frac{1}{2} \latticeenumerator{K}.
    \end{align*}
    Thus $\latticeenumerator{K} \leq 7$.
    
    Next  we assume thus that $\vnull \in \interior{P}$. Applying the theorems of Ehrhart and Pick
    again gives that
    \begin{align} \label{equation:pickehrhart} 
        \begin{split}
            \frac{9}{2} 
            &\geq \volume{K} 
            \geq \volume{P} 
            = \latticeenumerator{P} - \frac{1}{2}\latticeenumerator{\boundary{P}} -1 \\
            &= \latticeenumerator{K} - \frac{1}{2}(\latticeenumerator{K}-1) - 1
            = \frac{1}{2}\latticeenumerator{K} - \frac{1}{2},
        \end{split}
    \end{align}
    and thus $\latticeenumerator{K} \leq 10$. 
    
    Now, let $\latticeenumerator{K} = 10$. Then \eqref{equation:pickehrhart} implies
    that $\volume{K}=\volume{P}$ and thus $P=K$ by compactness. 
    Furthermore, we know that $P$ contains a lattice point
    in its interior as $\latticeenumerator{P}=10$. Scott's
    Theorem implicates that $P=K$ is unimodularly equivalent to $\ehrhartsimplex{2}$.
\end{proof}


\begin{bibdiv}
\begin{biblist}

\bib{MR2967480}{article}{
      author={Averkov, Gennadiy},
       title={On the size of lattice simplices with a single interior lattice
  point},
        date={2012},
        ISSN={0895-4801},
     journal={SIAM J. Discrete Math.},
      volume={26},
      number={2},
       pages={515\ndash 526},
         url={http://dx.doi.org/10.1137/110829052},
      review={\MR{2967480}},
}

\bib{MR3318147}{article}{
      author={Averkov, Gennadiy},
      author={Kr{\"u}mpelmann, Jan},
      author={Nill, Benjamin},
       title={Largest integral simplices with one interior integral point:
  {S}olution of {H}ensley's conjecture and related results},
        date={2015},
        ISSN={0001-8708},
     journal={Adv. Math.},
      volume={274},
       pages={118\ndash 166},
         url={http://dx.doi.org/10.1016/j.aim.2014.12.035},
      review={\MR{3318147}},
}

\bib{MR2271992}{book}{
      author={Beck, Matthias},
      author={Robins, Sinai},
       title={Computing the continuous discretely},
      series={Undergraduate Texts in Mathematics},
   publisher={Springer, New York},
        date={2007},
        ISBN={978-0-387-29139-0; 0-387-29139-3},
        note={Integer-point enumeration in polyhedra},
      review={\MR{2271992 (2007h:11119)}},
}

\bib{2012arXiv1204.1308B}{article}{
      author={{Berman}, R.~J.},
      author={{Berndtsson}, B.},
       title={{The volume of K\"ahler-Einstein Fano varieties and convex
  bodies}},
        date={2012-04},
     journal={ArXiv e-prints},
      eprint={1204.1308},
}

\bib{MR1194034}{article}{
      author={Betke, U.},
      author={Henk, M.},
      author={Wills, J.~M.},
       title={Successive-minima-type inequalities},
        date={1993},
        ISSN={0179-5376},
     journal={Discrete Comput. Geom.},
      volume={9},
      number={2},
       pages={165\ndash 175},
         url={http://dx.doi.org/10.1007/BF02189316},
      review={\MR{1194034 (93j:52026)}},
}

\bib{MR1316393}{book}{
      author={Croft, Hallard~T.},
      author={Falconer, Kenneth~J.},
      author={Guy, Richard~K.},
       title={Unsolved problems in geometry},
      series={Problem Books in Mathematics},
   publisher={Springer-Verlag, New York},
        date={1994},
        ISBN={0-387-97506-3},
        note={Corrected reprint of the 1991 original [ MR1107516 (92c:52001)],
  Unsolved Problems in Intuitive Mathematics, II},
      review={\MR{1316393 (95k:52001)}},
}

\bib{MR3022127}{article}{
      author={Draisma, Jan},
      author={McAllister, Tyrrell~B.},
      author={Nill, Benjamin},
       title={Lattice-width directions and {M}inkowski's {$3^d$}-theorem},
        date={2012},
        ISSN={0895-4801},
     journal={SIAM J. Discrete Math.},
      volume={26},
      number={3},
       pages={1104\ndash 1107},
         url={http://dx.doi.org/10.1137/120877635},
      review={\MR{3022127}},
}

\bib{MR0066430}{article}{
      author={Ehrhart, E.},
       title={Une g\'en\'eralisation du th\'eor\`eme de {M}inkowski},
        date={1955},
     journal={C. R. Acad. Sci. Paris},
      volume={240},
       pages={483\ndash 485},
      review={\MR{0066430 (16,574b)}},
}

\bib{MR0130860}{article}{
      author={Ehrhart, E.},
       title={Sur les poly\`edres rationnels homoth\'etiques \`a {$n$}\
  dimensions},
        date={1962},
     journal={C. R. Acad. Sci. Paris},
      volume={254},
       pages={616\ndash 618},
      review={\MR{0130860 (24 \#A714)}},
}

\bib{MR0213320}{article}{
      author={Ehrhart, E.},
       title={Sur un probl\`eme de g\'eom\'etrie diophantienne lin\'eaire. {I}.
  {P}oly\`edres et r\'eseaux},
        date={1967},
        ISSN={0075-4102},
     journal={J. Reine Angew. Math.},
      volume={226},
       pages={1\ndash 29},
      review={\MR{0213320 (35 \#4184)}},
}

\bib{MR0234919}{article}{
      author={Ehrhart, E.},
       title={Sur la loi de r\'eciprocit\'e des poly\`edres rationnels},
        date={1968},
     journal={C. R. Acad. Sci. Paris S\'er. A-B},
      volume={266},
       pages={A696\ndash A697},
      review={\MR{0234919 (38 \#3233)}},
}

\bib{MR518864}{article}{
      author={Ehrhart, E.},
       title={Volume r\'eticulaire critique d'un simplexe},
        date={1979},
        ISSN={0075-4102},
     journal={J. Reine Angew. Math.},
      volume={305},
       pages={218\ndash 220},
         url={http://dx.doi.org/10.1515/crll.1979.305.218},
      review={\MR{518864 (80b:52022)}},
}

\bib{MR0163219}{article}{
      author={Ehrhart, Eug{\`e}ne},
       title={Une g\'en\'eralisation probable du th\'eor\`eme fondamental de
  {M}inkowski},
        date={1964},
     journal={C. R. Acad. Sci. Paris},
      volume={258},
       pages={4885\ndash 4887},
      review={\MR{0163219 (29 \#522)}},
}

\bib{MR2335496}{book}{
   author={Gruber, Peter M.},
   title={Convex and discrete geometry},
   series={Grundlehren der Mathematischen Wissenschaften [Fundamental
   Principles of Mathematical Sciences]},
   volume={336},
   publisher={Springer, Berlin},
   date={2007},
   pages={xiv+578},
   isbn={978-3-540-71132-2},
   review={\MR{2335496 (2008f:52001)}},
}

\bib{MR0124818}{article}{
      author={Gr{\"u}nbaum, B.},
       title={Partitions of mass-distributions and of convex bodies by
  hyperplanes.},
        date={1960},
        ISSN={0030-8730},
     journal={Pacific J. Math.},
      volume={10},
       pages={1257\ndash 1261},
      review={\MR{0124818 (23 \#A2128)}},
}


\bib{HHZ}{article}{
author={Henk, Martin},
author={Henze, Matthias},
author={Hernandez Cifre, Maria},
title={Variations of Minkowski’s theorem on successive minima},
journal={Forum Math.}
note={(accepted for publication)}
}

\bib{MR688412}{article}{
      author={Hensley, Douglas},
       title={Lattice vertex polytopes with interior lattice points},
        date={1983},
        ISSN={0030-8730},
     journal={Pacific J. Math.},
      volume={105},
      number={1},
       pages={183\ndash 191},
         url={http://projecteuclid.org/euclid.pjm/1102723501},
      review={\MR{688412 (84c:52016)}},
}

\bib{MR1138580}{article}{
      author={Lagarias, Jeffrey~C.},
      author={Ziegler, G{\"u}nter~M.},
       title={Bounds for lattice polytopes containing a fixed number of
  interior points in a sublattice},
        date={1991},
        ISSN={0008-414X},
     journal={Canad. J. Math.},
      volume={43},
      number={5},
       pages={1022\ndash 1035},
         url={http://dx.doi.org/10.4153/CJM-1991-058-4},
      review={\MR{1138580 (92k:52032)}},
}

\bib{MR1764107}{article}{
      author={Milman, V.~D.},
      author={Pajor, A.},
       title={Entropy and asymptotic geometry of non-symmetric convex bodies},
        date={2000},
        ISSN={0001-8708},
     journal={Adv. Math.},
      volume={152},
      number={2},
       pages={314\ndash 335},
         url={http://dx.doi.org/10.1006/aima.1999.1903},
      review={\MR{1764107 (2001e:52004)}},
}

\bib{MR0249269}{book}{
      author={Minkowski, Hermann},
       title={Geometrie der {Z}ahlen},
      series={Bibliotheca Mathematica Teubneriana, Band 40},
   publisher={Johnson Reprint Corp., New York-London},
        date={1968},
      review={\MR{0249269 (40 \#2515)}},
}

\bib{MR3276123}{article}{
      author={Nill, Benjamin},
      author={Paffenholz, Andreas},
       title={On the equality case in {E}hrhart's volume conjecture},
        date={2014},
        ISSN={1615-715X},
     journal={Adv. Geom.},
      volume={14},
      number={4},
       pages={579\ndash 586},
         url={http://dx.doi.org/10.1515/advgeom-2014-0001},
      review={\MR{3276123}},
}

\bib{Pick}{article}{
      author={Pick, G.A.},
       title={Geometrisches zur {Z}ahlenlehre.},
        date={1899},
     journal={Sitzungsber. Lotos (Prag)},
      volume={19},
       pages={311\ndash 319},
}

\bib{Pikhurko2}{article}{
      author={Pikhurko, Oleg},
       title={Lattice points inside lattice polytopes},
        date={2000},
      eprint={arXiv:math/0008028v2},
         url={http://arxiv.org/abs/math/0008028v2},
}

\bib{MR1996360}{article}{
      author={Pikhurko, Oleg},
       title={Lattice points in lattice polytopes},
        date={2001},
        ISSN={0025-5793},
     journal={Mathematika},
      volume={48},
      number={1-2},
       pages={15\ndash 24 (2003)},
         url={http://dx.doi.org/10.1112/S0025579300014339},
      review={\MR{1996360 (2004f:52009)}},
}

\bib{MR0430960}{article}{
      author={Scott, P.~R.},
       title={On convex lattice polygons},
        date={1976},
        ISSN={0004-9727},
     journal={Bull. Austral. Math. Soc.},
      volume={15},
      number={3},
       pages={395\ndash 399},
      review={\MR{0430960 (55 \#3964)}},
}

\bib{MR0097771}{article}{
      author={Stewart, B.~M.},
       title={Asymmetry of a plane convex set with respect to its centroid},
        date={1958},
        ISSN={0030-8730},
     journal={Pacific J. Math.},
      volume={8},
       pages={335\ndash 337},
      review={\MR{0097771 (20 \#4238)}},
}

\bib{MR3300714}{article}{
      author={Taschuk, Steven},
       title={The harmonic mean measure of symmetry for convex bodies},
        date={2015},
        ISSN={1615-715X},
     journal={Adv. Geom.},
      volume={15},
      number={1},
       pages={109\ndash 120},
         url={http://dx.doi.org/10.1515/advgeom-2014-0021},
      review={\MR{3300714}},
}

\bib{MR651251}{article}{
      author={Zaks, J.},
      author={Perles, M.~A.},
      author={Wills, J.~M.},
       title={On lattice polytopes having interior lattice points},
        date={1982},
        ISSN={0013-6018},
     journal={Elem. Math.},
      volume={37},
      number={2},
       pages={44\ndash 46},
      review={\MR{651251 (83d:52012)}},
}

\end{biblist}
\end{bibdiv}
\end{document}